\documentclass[12pt]{amsart}

\usepackage{amssymb, amscd, amsthm, mathrsfs}

\newtheorem{Thm}{Theorem}[section]
\newtheorem{Prop}[Thm]{Proposition}
\newtheorem{Cor}[Thm]{Corollary}

\newtheorem{Conj}[Thm]{Conjecture}
\theoremstyle{remark}

\newtheorem*{Notation}{Notation}
\newtheorem*{Ack}{Acknowledgement}

\newcommand{\Order}{\mathcal{O}}
\newcommand{\into}{\hookrightarrow}

\newcommand{\isomto}{\overset{\sim}{\to}}
\newcommand{\isomfrom}{\overset{\sim}{\leftarrow}}

\newcommand{\tensor}{\mathbin{\otimes}}
\newcommand{\closure}[1]{\overline{#1}}
\newcommand{\N}{\mathbb{N}}
\newcommand{\Z}{\mathbb{Z}}
\newcommand{\Q}{\mathbb{Q}}
\newcommand{\R}{\mathbb{R}}
\newcommand{\F}{\mathbb{F}}
\newcommand{\et}{\mathrm{et}}
\newcommand{\fppf}{\mathrm{fppf}}

\newcommand{\cont}{\mathrm{cont}}
\newcommand{\id}{\mathrm{id}}
\newcommand{\Gm}{\mathbf{G}_{m}}
\newcommand{\Ga}{\mathbf{G}_{a}}
\newcommand{\tor}{\mathrm{tor}}
\newcommand{\divis}{\mathrm{div}}

\newcommand{\invlim}{\varprojlim}
\newcommand{\vol}{\mathrm{vol}}
\newcommand{\pairing}[2]{\langle #1, #2 \rangle}
\newcommand{\adele}{\mathbb{A}}
\newcommand{\sep}{\mathrm{sep}}
\newcommand{\ur}{\mathrm{ur}}
\newcommand{\Mod}{\mathrm{Mod}}
\newcommand{\Ind}{\mathrm{Ind}}
\newcommand{\Pro}{\mathrm{Pro}}

\DeclareMathOperator{\Hom}{Hom}

\DeclareMathOperator{\Spec}{Spec}
\DeclareMathOperator{\Ab}{Ab}

\DeclareMathOperator{\Pic}{Pic}

\DeclareMathOperator{\Lie}{Lie}
\DeclareMathOperator{\Disc}{Disc}
\DeclareMathOperator{\rank}{rank}
\let\Re\relax
\DeclareMathOperator{\Re}{Re}

\DeclareFontFamily{U}{wncy}{}
\DeclareFontShape{U}{wncy}{m}{n}{<->wncyr10}{}
\DeclareSymbolFont{mcy}{U}{wncy}{m}{n}
\DeclareMathSymbol{\Sha}{\mathord}{mcy}{"58}

\title[A Weil-\'etale version of the BSD formula]
	{A Weil-\'etale version of the Birch and Swinnerton-Dyer formula over function fields}
\author{Thomas H. Geisser}
\address{
	Rikkyo University, Ikebukuro, Tokyo, Japan
}
\email{geisser@rikkyo.ac.jp}
\author{Takashi Suzuki}
\address{
	Department of Mathematics, Chuo University,
	1-13-27 Kasuga, Bunkyo-ku, Tokyo 112-8551, JAPAN
}
\email{tsuzuki@gug.math.chuo-u.ac.jp}
\thanks{
	The first named author is supported by JSPS Grant-in-Aid (A) 15H02048-1, (C) 18K03258.
	The second named author is a Research Fellow of Japan Society for the Promotion of Science
	and supported by JSPS KAKENHI Grant Number JP18J00415.
}
\date{September 14, 2019}

\begin{document}

\begin{abstract}
We give a reformulation of the Birch and Swinnerton-Dyer conjecture over
global function fields in terms of Weil-\'etale cohomology of the curve with coefficients in the N\'eron
model, and show that it holds under the assumption of
finiteness of the Tate-Shafarevich group.
\end{abstract}

\maketitle


\section{Introduction}
\label{sec: Introduction}

The conjecture of Birch and Swinnerton-Dyer is one of the most
important problems in arithmetic geometry. It states that the rank 
of the rational points $A(K)$ 
of an abelian variety $A$ over a global field $K$ is equal to the order of 
vanishing of the $L$-function $L(A,s)$ associated with $A$ at $s=1$, 
and relates the
leading term of the $L$-function to various invariants associated with $A$.
If $K$ has characteristic $p$ and under the assumption of
finiteness of the Tate-Shafarevich group $\Sha(A)$, Schneider \cite{Sch82} proved a formula
for the prime to $p$-part of the leading coefficient, Bauer \cite{Bau92} gave a formula
in case $A$ has good reduction at every place, and Kato-Trihan \cite{KT03} proved
a formula in general. 

In this paper we give a formula for the leading coefficient in terms of
Weil-\'etale cohomology $H_{W}^{\ast}(S, \mathcal{A})$ of the regular complete curve $S$ with function field $K$
with coefficients in the N\'eron model $\mathcal{A}$ of $A$.
More precisely,
let $\F_{q}$ be the field of constants of $S$
and $e \in H_{W}^{1}(\F_{q}, \Z)$ the element
corresponding to the $q$-th power arithmetic Frobenius morphism.
The cup product with $e$ defines a complex
	\[
			\cdots
		\overset{e}{\to}
			H_{W}^{i}(S, \mathcal{A})
		\overset{e}{\to}
			H_{W}^{i + 1}(S, \mathcal{A})
		\overset{e}{\to}
			H_{W}^{i + 2}(S, \mathcal{A})
		\overset{e}{\to}
			\cdots,
	\]
whose cohomology groups are finite
if the groups $H_{W}^{\ast}(S, \mathcal{A})$ are finitely generated.
In this case, we denote the alternating product of their orders by $\chi(H_{W}^{\ast}(S, \mathcal{A}), e)$.
Let $\Lie(\mathcal{A})$ be the vector bundle on $S$
defined by the pullback of the dual of $\Omega_{\mathcal{A} / S}^{1}$
by the zero section $S \into \mathcal{A}$,
and $\chi(S, \Lie(\mathcal{A}))$ the alternating sum of dimensions of
the coherent cohomology $H^{\ast}(S, \Lie(\mathcal{A}))$ over $\F_{q}$.

\begin{Thm} \label{thm: BSD by Weil etale}
	Let $A$ be an abelian variety over a global field $K$ of 
	characteristic $p$ and assume that $\Sha(A)$ is finite.
	Then the rank $r$ of $A(K)$ equals the order of vanishing of $L(A,s)$ at $s=1$,
	the groups $H_{W}^{\ast}(S, \mathcal{A})$ are finitely generated, and
		\[
				\lim_{s \to 1}
					\frac{L(A, s)}{(s - 1)^{r}}
			=
				\chi(H_{W}^{\ast}(S, \mathcal{A}), e)^{-1}
				\cdot q^{\chi(S, \Lie(\mathcal{A}))}\cdot (\log q)^{r}.
		\]
\end{Thm}

The same statements holds if we replace $\mathcal A$ by $\mathcal A^0$,
the subgroup scheme with all fibers connected. Our result fits into the
general philosophy that important conjectures in arithmetic 
geometry are equivalent to finite generation statements of Weil-\'etale
cohomology groups, and special values of zeta and $L$-functions can
be expressed as Euler characteristics of Weil-\'etale cohomology
(\cite{Lic05}, \cite{Gei04}).

The proof proceeds by showing that finiteness of $\Sha(A)$ is equivalent
to finite generation of the groups $H_{W}^i(S, \mathcal{A})$, and implies
an identity
	\[
			\chi(H_{W}^{\ast}(S, \mathcal{A}), e)^{-1}
		=
			\frac{
				\# \Sha(A)
			}{
				\# A(K)_{\tor} \cdot \# B(K)_{\tor}
			}
		\cdot
			\frac{
				\Disc(h)
			}{
				(\log q)^{r}
			}
		\cdot
			c(A),
	\]
where $B$ is the abelian variety dual to $A$, $h$ the height pairing,
and $c(A)$ the product of the orders of the groups of $k(v)$-rational connected components of 
the fibers of the N\'eron model over all places $v$.
Key ingredients are results of the second named author \cite{Suz19}.
Theorem \ref{thm: BSD by Weil etale} follows then by applying the 
result of Kato-Trihan \cite[Chap.\ I, Thm.]{KT03}. 

At the end of the paper we show that Weil-\'etale cohomology is
an integral model of $l$-adic cohomology:

\begin{Thm}
	Assume that $\Sha(A)$ is finite.
	Let $l$ be a prime
	and ${}_{l^{i}}(\mathcal{A}^{0})$ the $l^{i}$-torsion part of $\mathcal{A}^{0}$.
	If $l \ne p$, then the canonical homomorphism
		\[
				H_{W}^{n}(S, \mathcal{A}^{0}) \tensor \Z_{l}
			\to
				\invlim_{i}
					H_{\et}^{n + 1}(S, {}_{l^{i}}(\mathcal{A}^{0}))
		\]
	is an isomorphism.
	If $l = p$ and $A$ has semistable reduction everywhere,
	then the canonical homomorphism
		\[
				H_{W}^{n}(S, \mathcal{A}^{0}) \tensor \Z_{p}
			\to
				\invlim_{i}
					H_{\fppf}^{n + 1}(S, {}_{p^{i}}(\mathcal{A}^{0}))
		\]
	is an isomorphism.
\end{Thm}

For more complete results (without the semistability assumption for the case $l = p$),
see Prop.\ \ref{prop: integral models in derived setting} and \ref{prop: integral models in abelian setting}.

\begin{Ack}
	The authors thank Fabien Trihan for helpful discussions.
\end{Ack}

\begin{Notation}
	Throughout this paper,
	we fix a finite field $\F_{q}$ with $q$ elements of characteristic $p$,
	an algebraic closure $\closure{\F_{q}}$ of $\F_{q}$,
	a proper smooth geometrically connected curve $S$ over $\F_{q}$ with function field $K$,
	and an abelian variety $A$ over $K$ of dimension $d$.
	We denote the rank of the Mordell-Weil group $A(K)$ by $r$.
	Thus
		\[
			d = \dim(A), \quad
			r = \rank(A(K)).
		\]
	
	For a place $v$ of $K$ (i.e.\ a closed point of $S$),
	we denote the residue field of $S$ at $v$ by $k(v)$, its cardinality by $N(v)$,
	the completed local ring by $\Order_{v}$,
	its maximal ideal by $m_{v}$ and its fraction field by $K_{v}$.
	The adele ring of $K$ is denoted by $\adele_{K}$
	and the integral adele ring by $\Order_{\adele_{K}}$.
	For a group scheme $G$ over $\Order_{v}$ and $n \ge 0$,
	we denote the kernel of the reduction map $G(\Order_{v}) \to G(\Order_{v} / m_{v}^{n})$
	by $G(m_{v}^{n})$.
	
	The N\'eron model over $S$ of $A / K$ is denoted by $\mathcal{A}$.
	The open subgroup scheme of $\mathcal{A}$ with connected fibers is denoted by $\mathcal{A}^{0}$
	and the fiber of $\mathcal{A}$ at a closed point $v \in S$ is denoted by $\mathcal{A}_{v}$.
	Let $B$ be the dual abelian variety of $A / K$.
	We have objects $\mathcal{B}$, $\mathcal{B}^{0}$, $\mathcal{B}_{v}$ correspondingly.
	
	For an abelian group $G$, its torsion part is denoted by $G_{\tor}$
	and torsion-free quotient by $G / \tor$.
	The $n$-torsion part (kernel of multiplication by $n$) of $G$ for an integer $n$ is denoted by ${}_{n} G$.
	For a pairing $\varphi$ between finitely generated abelian groups $G$ and $H$
	with values in $\Z$, $\Q$ or $\R$,
	we denote by $\Disc(\varphi)$
	the absolute value of the determinant of the matrix presentation of $\varphi$
	with respect to some (or equivalently, any) $\Z$-bases of $G / \tor$ and $H / \tor$.
	
	If $\mathcal{F}$ is a coherent sheaf on $S$,
	the Euler characteristic of $\mathcal{F}$ is
		\[
				\chi(S, \mathcal{F})
			=
				\sum_{i}
					(-1)^{i} \dim_{\F_{q}} H^{i}(S, \mathcal{F}).
		\]
	If $C^{\cdot}$ is a complex of abelian groups with finitely many finite cohomology groups,
	then we denote
		\[
				\chi(C^{\cdot})
			=
				\prod_{i}
					(\# H^{i}(C^{\cdot}))^{(-1)^{i}}.
		\]
	If $C^{\cdot}$ is a graded object of finite abelian groups with finitely many terms,
	then we denote
		\[
				\chi(C^{\cdot})
			=
				\prod_{i} (\# C^{i})^{(-1)^{i}}.
		\]
	These two pieces of notation are compatible
	by viewing a graded object as a complex with zero differentials.
\end{Notation}


\section{The Birch and Swinnerton-Dyer formulas by Tate and by Kato-Trihan}
\label{sec: The Birch and Swinnerton-Dyer formulas by Tate and by Kato Trihan}

We recall the Birch and Swinnerton-Dyer conjecture for $A / K$
as formulated by Tate \cite{Tat68} and by Kato-Trihan \cite{KT03}.

Let $l \ne p$ be a prime number.
The rational $l$-adic Tate module $V_{l}(A) = T_{l}(A) \tensor \Q$ is dual to
the $l$-adic cohomology $H_{\cont}^{1}(A \times_{K} K^{\sep}, \Q_{l})$
as Galois representations over $K$,
where $K^{\sep}$ is a separable closure of $K$.
It is also dual to the negative Tate twist $V_{l}(B)(-1)$ of $V_{l}(B)$
via the Weil pairing,
where $B$ is the dual abelian variety of $A$ as in Notation at the end of Introduction.
Since $B$ is (non-canonically) isogenous to $A$,
$V_{l}(B)$ is (non-canonically) isomorphic to $V_{l}(A)$.
Summarizing,
	\begin{equation} \label{eq: isomorphisms of l adic representations}
		\begin{split}
			&		H_{\cont}^{1}(A \times_{K} K^{\sep}, \Q_{l})
				=
					\Hom(V_{l}(A), \Q_{l})
			\\
			&	=
					V_{l}(B)(-1)
				\cong
					V_{l}(A)(-1).
		\end{split}
	\end{equation}

For a place $v$ of $K$ where $A$ has good reduction,
we define a polynomial in $t$ by
	\[
			P_{v}(t)
		=
			\det \bigl(
					1 - \varphi_{v} t
				\bigm|
					H_{\cont}^{1}(A \times_{K} K^{\sep}, \Q_{l})
			\bigr),
	\]
where $\varphi_{v}$ is the geometric Frobenius at $v$.
This has $\Z$-coefficients and does not depend on $l$.
Let $U \subset S$ be an open dense subscheme where $A$ has good reduction.
Denote $\Sigma = S \setminus U$.
As in \cite[\S 1.3]{KT03},
for a complex number $s$ with $\Re(s) > 3 / 2$,
we define the $L$-function $L(U, A, s)$ without Euler factors outside $U$ by
	\[
			L(U, A, s)
		=
			\prod_{v \in U}
				P_{v}(N(v)^{-s})^{-1}.
	\]
This is a rational function in $q^{- s}$ and regular at $s = 1$.
In \cite[(1.3)]{Tat68}, this is denoted by $L_{\Sigma}(s)$.

Let $\Lie(\mathcal{A}) / S$ be the Lie algebra (with zero Lie bracket)
of the N\'eron model $\mathcal{A} / S$ (\cite[Exp.\ II, 4.11]{DG70}).
It is a group scheme represented by a locally free sheaf on $S$ of rank $d = \dim(A)$
that is given by the pullback along the zero section $S \into \mathcal{A}$
of the $\Order_{\mathcal{A}}$-linear dual of $\Omega_{\mathcal{A} / S}^{1}$.
We similarly have the Lie algebra $\Lie(A) / K$ of $A / K$.

For each place $v$,
we give $K_{v}$ the normalized Haar measure $\mu_{v}$,
so that $\mu_{v}(\Order_{v}) = 1$.
(The measure $\mu_{v}$ in \cite[\S 1]{Tat68} is not necessarily normalized for all $v$.
Our normalization is just for simplicity.)
The product $\mu = \prod_{v} \mu_{v}$ gives a Haar measure on the adele ring $\adele_{K}$.
In \cite[(1.5)]{Tat68},
the measure $\mu(\adele_{K} / K)$ of the compact quotient $\adele_{K} / K$ is denoted by $|\mu|$.

We fix a non-zero invariant top degree differential form $\omega$ on $A / K$.
As in \cite[\S 1]{Tat68},
$\omega$ and $\mu_{v}$ on $K_{v}$ together
determine a Haar measure on $\Lie(A)(K_{v})$ (still denoted by $\mu_{v}$) for each $v$.
By definition \cite[\S 2.2.1]{Wei82},
this measure is characterized by
	\begin{equation} \label{eq: Haar measure for Lie algebra}
			\mu_{v}(L)
		=
			[\omega(\det(L)) : \Order_{v}]
	\end{equation}
for any full rank $\Order_{v}$-lattice $L$ of $\Lie(A)(K_{v})$,
where $\det(L)$ is the top exterior power of $L$ over $\Order_{v}$,
$\omega(\det(L)) \subset K_{v}$ is the image of $\det(L)$
by $\omega$ viewed as a $K_{v}$-linear isomorphism $\det(\Lie(A)(K_{v})) \isomto K_{v}$,
and the index $[\omega(\det(L)) : \Order_{v}]$ means
$[\Order_{v} : \omega(\det(L))]^{-1}$ in case
$\omega(\det(L))$ does not contain $\Order_{v}$.
One may take this formula as the definition of $\mu_{v}$.
If $\omega$ is replaced by its multiple by a rational function $f \in K^{\times}$,
then $\mu_{v}$ is multiplied by the normalized absolute value $|f|_{v} = N(v)^{-v(f)}$.

We have $\mu_{v}(\Lie(\mathcal{A})(\Order_{v})) = 1$ for almost all $v$.
Hence the product $\mu = \prod_{v} \mu_{v}$ defines
a Haar measure on the adelic points $\Lie(A)(\adele_{K})$.
The measure $\mu_{v}$ on $\Lie(A)(K_{v})$ in turn determines a Haar measure on $A(K_{v})$
such that
	\begin{equation} \label{eq: Haar measure of group and Lie algebra}
			\mu_{v}(\Lie(\mathcal{A})(m_{v}^{n}))
		=
			\mu_{v}(\mathcal{A}(m_{v}^{n})))
	\end{equation}
for all $n \ge 1$.
The measure $\mu_{v}(A(K_{v}))$ is denoted by
$\int_{A(K_{v})} |\omega|_{v} \mu_{v}^{d}$ in \cite[\S 1]{Tat68}.

Assume that $U \subset S$ is large enough
so that $\omega$ gives a nowhere vanishing section of the dual of the line bundle $\det(\Lie(\mathcal{A}))$ over $U$,
in addition that $A$ has good reduction over $U$.
Following \cite[(1.5)]{Tat68}, we define
	\[
			L_{\Sigma}^{\ast}(s)
		=
			L_{\Sigma}^{\ast}(A, s)
		=
				\frac{
					\mu(\adele_{K} / K)^{d}
				}{
					\prod_{v \not\in U}
						\mu_{v}(A(K_{v}))
				}
			\cdot
				L(U, A, s).
	\]
This is independent of the choice of $\omega$ by the product formula.
As shown after loc.\ cit.,
the asymptotic behavior of $L_{\Sigma}^{\ast}(A, s)$ as $s \to 1$ does not depend on $U$ (or $\Sigma$).
Also, following \cite[\S 1.7]{KT03}, we define
	\[
			\vol \Bigl(
				\prod_{v \not\in U} A(K_{v})
			\Bigr)
		=
			\frac{
				\prod_{v \not\in U} \mu_{v}(A(K_{v}))
			}{
				\mu \bigl(
					\Lie(A)(\adele_{K}) / \Lie(A)(K)
				\bigr)
			},
	\]
which is independent of the choice of $\omega$.

Let $\Sha(A)$ be the Tate-Shafarevich group of $A$
and
	\[
			h
		\colon
			A(K) \times B(K)
		\to
			\R
	\]
the N\'eron-Tate height pairing.
We have $\Disc(h) \ne 0$.

Now Tate's formulation of the Birch and Swinnerton-Dyer conjecture is the following.

\begin{Conj}[{\cite[\S 1, (A), (B)]{Tat68}}] \label{conj: Tate BSD}
	The order of zero of $L(U, A, s)$ at $s = 1$ is the Mordell-Weil rank $r$.
	The group $\Sha(A)$ is finite.
	We have
		\[
				\lim_{s \to 1}
					\frac{
						L_{\Sigma}^{\ast}(A, s)
					}{
						(s - 1)^{r}
					}
			=
				\frac{
					\# \Sha(A) \cdot \Disc(h)
				}{
					\# A(K)_{\tor} \cdot \# B(K)_{\tor}
				}.
		\]
\end{Conj}

On the other hand, Kato-Trihan's formulation is the following.

\begin{Conj}[{\cite[1.3.1, 1.4.1, 1.8.1]{KT03}}] \label{conj: Kato Trihan BSD}
	The order of zero of $L(U, A, s)$ at $s = 1$ is $r$.
	The group $\Sha(A)$ is finite.
	We have
		\[
				\lim_{s \to 1}
					\frac{
						L(U, A, s)
					}{
						(s - 1)^{r}
					}
			=
					\frac{
						\# \Sha(A) \cdot \Disc(h)
					}{
						\# A(K)_{\tor} \cdot \# B(K)_{\tor}
					}
				\cdot
					\vol \Bigl(
						\prod_{v \not\in U} A(K_{v})
					\Bigr).
		\]
\end{Conj}


\section{Comparison of the two formulas}
\label{sec: Comparison of the two formulas}

In this section,
we show, for the convenience of the reader,
that Conj.\ \ref{conj: Tate BSD} and \ref{conj: Kato Trihan BSD} are equivalent.
It suffices to show (without hypothesis on the order of zero of $L(U, A, s)$ or finiteness of $\Sha(A)$) the following.

\begin{Prop} \label{prop: comparison of measure terms of Tate and Kato Trihan}
	\[
			\mu \left(
				\frac{
					\Lie(A)(\adele_{K})
				}{
					\Lie(A)(K)
				}
			\right)
		=
			\mu(\adele_{K} / K)^{d}.
	\]
\end{Prop}

We are going to reduce this to the Riemann-Roch theorem
for the vector bundle $\Lie(\mathcal{A})$ over $S$.
First we relate the left-hand side to the Euler characteristic $\chi(S, \Lie(\mathcal{A}))$.
This step will also be used in the next section
as one of the keys for the proof of Thm.\ \ref{thm: BSD by Weil etale}.

\begin{Prop} \label{prop: measure of Lie ideles and Euler char}
	\[
			\mu \left(
				\frac{
					\Lie(A)(\adele_{K})
				}{
					\Lie(A)(K)
				}
			\right)
		=
				\mu \bigl(
					\Lie(\mathcal{A})(\Order_{\adele_{K}})
				\bigr)
			\cdot
				q^{- \chi(S, \Lie(\mathcal{A}))}.
	\]
\end{Prop}

\begin{proof}
	Let $\Order_{(v)}$ be the (Zariski) local ring of $S$ at $v$.
	The excision and localization sequences for Zariski cohomology give long exact sequences
		\begin{gather*}
				\cdots
			\to
				\bigoplus_{v}
					H_{v}^{n}(\Order_{(v)}, \Lie(\mathcal{A}))
			\to
				H^{n}(S, \Lie(\mathcal{A}))
			\to
				H^{n}(K, \Lie(A))
			\to
				\cdots,
			\\
				\cdots
			\to
				H_{v}^{n}(\Order_{(v)}, \Lie(\mathcal{A}))
			\to
				H^{n}(\Order_{(v)}, \Lie(\mathcal{A}))
			\to
				H^{n}(K, \Lie(A))
			\to
				\cdots
		\end{gather*}
	(for each place $v$ of $K$ in the latter sequence),
	where $H_{v}^{n}$ denotes cohomology with closed support.
	The (Zariski) cohomology groups
	$H^{n}(K, \Lie(A))$ and
	$H^{n}(\Order_{(v)}, \Lie(\mathcal{A}))$ are zero for $n \ge 1$.
	Hence the finite groups $H^{0}(S, \Lie(\mathcal{A}))$ and $H^{1}(S, \Lie(\mathcal{A}))$
	are given by the kernel and the cokernel, respectively, of the natural homomorphism
		\[
				\Lie(A)(K)
			\to
				\bigoplus_{v}
					\frac{
						\Lie(A)(K)
					}{
						\Lie(\mathcal{A})(\Order_{(v)})
					}.
		\]
	Each summand of the right-hand side is isomorphic to
	$\Lie(A)(K_{v}) / \Lie(\mathcal{A})(\Order_{v})$
	by approximation.
	Therefore the right-hand side (the whole direct sum) is isomorphic to
	$\Lie(A)(\adele_{K}) / \Lie(\mathcal{A})(\Order_{\adele_{K}})$.
	In other words, we have a natural exact sequence
		\[
				0
			\to
				\frac{
					\Lie(\mathcal{A})(\Order_{\adele_{K}})
				}{
					H^{0}(S, \Lie(\mathcal{A}))
				}
			\to
				\frac{
					\Lie(A)(\adele_{K})
				}{
					\Lie(A)(K)
				}
			\to
				H^{1}(S, \Lie(\mathcal{A}))
			\to
				0.
		\]
	Thus
		\begin{align*}
					\mu \left(
						\frac{
							\Lie(A)(\adele_{K})
						}{
							\Lie(A)(K)
						}
					\right)
			&	=
						\mu \left(
							\frac{
								\Lie(\mathcal{A})(\Order_{\adele_{K}})
							}{
								H^{0}(S, \Lie(\mathcal{A}))
							}
						\right)
					\cdot
						\# H^{1}(S, \Lie(\mathcal{A}))
			\\
			&	=
						\mu \bigl(
							\Lie(\mathcal{A})(\Order_{\adele_{K}})
						\bigr)
					\cdot
						\frac{
							\# H^{1}(S, \Lie(\mathcal{A}))
						}{
							\# H^{0}(S, \Lie(\mathcal{A}))
						}
			\\
			&	=
						\mu \bigl(
							\Lie(\mathcal{A})(\Order_{\adele_{K}})
						\bigr)
					\cdot
						q^{- \chi(S, \Lie(\mathcal{A}))}.
		\end{align*}
\end{proof}

The above proposition is an intermediate step,
as the term
	$
		\mu \bigl(
			\Lie(\mathcal{A})(\Order_{\adele_{K}})
		\bigr)
	$
can be more explicitly calculated as follows.

\begin{Prop}
	\[
			\mu \bigl(
				\Lie(\mathcal{A})(\Order_{\adele_{K}})
			\bigr)
		=
			q^{\deg(\det(\Lie(\mathcal{A})))}.
	\]
\end{Prop}

\begin{proof}
	For any $v$, let $v(\omega) \in \Z$ be the order of zero at $v$
	of the rational section $\omega$ of the dual of the line bundle $\det(\Lie(\mathcal{A}))$ over $S$.
	Then $\mu_{v}(\Lie(\mathcal{A})(\Order_{v})) = N(v)^{- v(\omega)}$
	by \eqref{eq: Haar measure for Lie algebra}.
	We have $- \sum_{v} [k(v) : \F_{q}] v(\omega) = \deg(\det(\Lie(\mathcal{A})))$.
	This gives the result.
\end{proof}

\begin{proof}[Proof of Prop.\ \ref{prop: comparison of measure terms of Tate and Kato Trihan}]
	By the previous two propositions, we have
		\[
				\mu \left(
					\frac{
						\Lie(A)(\adele_{K})
					}{
						\Lie(A)(K)
					}
				\right)
			=
				q^{\deg(\det(\Lie(\mathcal{A}))) - \chi(S, \Lie(\mathcal{A}))}.
		\]
	The same calculations, applied to the structure sheaf $\Order_{S}$ instead of $\Lie(\mathcal{A})$,
	show that
		\[
				\mu(\adele_{K} / K)
			=
					\mu(\Order_{\adele_{K}})
				\cdot
					q^{- \chi(S, \Order_{S})}
			=
				q^{- \chi(S, \Order_{S})}.
		\]
	Therefore the result follows from the Riemann-Roch theorem
		\begin{equation} \label{eq: Riemann Roch}
				\chi(S, \Lie(\mathcal{A}))
			=
					d \cdot \chi(S, \Order_{S})
				+
					\deg(\det(\Lie(\mathcal{A}))).
		\end{equation}
\end{proof}


\section{Bad Euler factors}
\label{sec: Bad Euler factors}

Using Prop.\ \ref{prop: measure of Lie ideles and Euler char} above,
we will rewrite the Birch and Swinnerton-Dyer formula in a form including bad Euler factors
and $\chi(S, \Lie(\mathcal{A}))$ without terms defined by Haar measures.

As in the previous section,
let $l \ne p$ be a prime number.
For any place $v$ of $K$ where $A$ may have good or bad reduction,
we define a polynomial in $t$ by
	\[
			P_{v}(t)
		=
			\det \bigl(
					1 - \varphi_{v} t
				\bigm|
					H_{\cont}^{1}(A \times_{K} K^{\sep}, \Q_{l})^{I_{v}}
			\bigr),
	\]
where $I_{v}$ is the inertial group at $v$.
We define the completed $L$-function by
	\[
			L(A, s)
		=
				L(U, A, s)
			\cdot
				\prod_{v \not\in U}
					P_{v}(N(v)^{-s})^{-1}
		=
			\prod_{v}
				P_{v}(N(v)^{-s})^{-1},
	\]
where the latter product is over all places $v$.
Recall that $d = \dim(A)$.

\begin{Prop} \label{prop: special value of Euler factor}
	For any place $v$,
	the polynomial $P_{v}(t)$ has $\Z$-coefficients and does not depend on $l \ne p$.
	We have
		\[
				P_{v}(N(v)^{-1})
			=
				\frac{
					\# \mathcal{A}^{0}(k(v))
				}{
					N(v)^{d}
				}.
		\]
\end{Prop}

\begin{proof}
	This is well-known.
	We recall its proof.
	By \eqref{eq: isomorphisms of l adic representations},
	we have
		\[
				P_{v}(t)
			=
				\det \bigl(
						1 - \varphi_{v} t
					\bigm|
						V_{l}(A)^{I_{v}}(-1)
				\bigr)
			=
				\det \bigl(
						1 - N(v) \varphi_{v} t
					\bigm|
						V_{l}(A)^{I_{v}}
				\bigr).
		\]
	Let $K_{v}^{\ur}$ be the maximal unramified extension of $K_{v}$
	with ring of integers $\Order_{v}^{\ur}$.
	Then
		\[
				V_{l}(A)^{I_{v}}
			=
				V_{l}(A)(K^{\ur})
			=
				V_{l}(A(K^{\ur}))
			=
				V_{l}(\mathcal{A}(\Order_{v}^{\ur})).
		\]
	By the smoothness of $\mathcal{A}$,
	the reduction map $\mathcal{A}(\Order_{v}^{\ur}) \to \mathcal{A}(\closure{k(v)})$ is surjective.
	Its kernel is uniquely $l$-divisible.
	With the finiteness of the component group
	$\pi_{0}(\mathcal{A}_{v}) = \mathcal{A}_{v} / \mathcal{A}_{v}^{0}$
	of the fiber $\mathcal{A}_{v}$, we have
	$V_{l}(A)^{I_{v}} \cong V_{l}(\mathcal{A}_{v}^{0})$
	as $l$-adic representations over $k(v)$.
	By the Chevalley structure theorem,
	the algebraic group $\mathcal{A}_{v}^{0}$ over $k(v)$
	has a canonical filtration whose graded pieces are a torus $T$, a smooth connected unipotent group $U$
	and an abelian variety $A'$.
	Since $U$ is $p$-power-torsion,
	we have an exact sequence
	$0 \to V_{l}(T) \to V_{l}(\mathcal{A}_{v}^{0}) \to V_{l}(A') \to 0$
	and an equality
		\[
				P_{v}(t)
			=
					\det \bigl(
							1 - N(v) \varphi_{v} t
						\bigm|
							V_{l}(T)
					\bigr)
				\cdot
					\det \bigl(
							1 - N(v) \varphi_{v} t
						\bigm|
							V_{l}(A')
					\bigr).
		\]
	
	On the other hand, a short exact sequence of connected algebraic groups over a finite field
	induces a short exact sequence of their groups of rational points
	by Lang's theorem.
	We have $d = \dim(T) + \dim(U) + \dim(A')$.
	Hence
		\[
				\frac{
					\# \mathcal{A}^{0}(k(v))
				}{
					N(v)^{d}
				}
			=
					\frac{
						\# T(k(v))
					}{
						N(v)^{\dim(T)}
					}
				\cdot
					\frac{
						\# U(k(v))
					}{
						N(v)^{\dim(U)}
					}
				\cdot
					\frac{
						\# A'(k(v))
					}{
						N(v)^{\dim(A')}
					}.
		\]
	The group $U$ is a finite successive extension of copies of $\Ga$.
	Hence the middle factor in the right-hand side is $1$.
	
	Therefore we may treat $T$ and $A'$ separately.
	The $A'$-factor is classical and treated by Weil
	(use \cite[(1.1), (1.2)]{Tat68}), resulting that the polynomial
		$
			\det \bigl(
					1 - N(v) \varphi_{v} t
				\bigm|
					V_{l}(A')
			\bigr)
		$
	has $\Z$-coefficients, does not depend on $l$ and
		\[
				\det \bigl(
						1 - \varphi_{v}
					\bigm|
						V_{l}(A')
				\bigr)
			=
				\frac{
					\# A'(k(v))
				}{
					N(v)^{\dim(A')}
				}.
		\]
	About the $T$-factor,
	we have $V_{l}(T) \cong \Hom(X^{\ast}(T), \Q_{l}(1))$ as $l$-adic representations,
	where $X^{\ast}(T)$ is the character group of $T$.
	Hence
		\[
				\det \bigl(
						1 - N(v) \varphi_{v} t
					\bigm|
						V_{l}(T)
				\bigr)
			=
				\det \bigl(
						1 - \varphi_{v}^{-1} t
					\bigm|
						X^{\ast}(T)
				\bigr).
		\]
	This has $\Z$-coefficients and does not depend on $l$.
	Its value at $t = N(v)^{-1}$ is
		\[
				\frac{
				\det \bigl(
						N(v) - \varphi_{v}^{-1}
					\bigm|
						X^{\ast}(T)
				\bigr)
				}{
					N(v)^{\dim(T)}
				}
			=
				\frac{
					\# T(k(v))
				}{
					N(v)^{\dim(T)}
				},
		\]
	where the last equality is \cite[I, 1.5]{Oes84}.
\end{proof}

We define the (global) Tamagawa factor of $A$ by
	\[
			c(A)
		=
			\prod_{v}
				\# \pi_{0}(\mathcal{A}_{v})(k(v)).
	\]
Let $U$ and $\omega$ be as in Conj.\ \ref{conj: Tate BSD}.

\begin{Prop} \label{prop: measure of local points}
	\[
			\prod_{v \not\in U}
				\mu_{v}(A(K_{v}))
		=
				c(A)
			\cdot
				\mu \bigl(
					\Lie(\mathcal{A})(\Order_{\adele_{K}})
				\bigr)
			\cdot
				\prod_{v \not\in U}
					P_{v}(N(v)^{-1}).
	\]
\end{Prop}

\begin{proof}
	For any place $v$,
	the reduction map $A(K_{v}) = \mathcal{A}(\Order_{v}) \to \mathcal{A}(k(v))$
	is surjective
	since $\mathcal{A}$ is smooth and $\Order_{v}$ is henselian.
	Therefore we have an exact sequence
		\[
				0
			\to
				\mathcal{A}(m_{v})
			\to
				A(K_{v})
			\to
				\mathcal{A}(k(v))
			\to
				0
		\]
	and an equality
		\[
				\mu_{v}(A(K_{v}))
			=
					\# \mathcal{A}(k(v))
				\cdot
					\mu_{v}(\mathcal{A}(m_{v})).
		\]
	By Lang's theorem, we have an exact sequence
		\[
				0
			\to
				\mathcal{A}^{0}(k(v))
			\to
				\mathcal{A}(k(v))
			\to
				\pi_{0}(\mathcal{A}_{v})(k(v))
			\to
				0
		\]
	and hence an equality
		\[
				\# \mathcal{A}(k(v))
			=
					\# \pi_{0}(\mathcal{A}_{v})(k(v))
				\cdot
					\# \mathcal{A}^{0}(k(v)).
		\]
	By \eqref{eq: Haar measure of group and Lie algebra}, we have
		\[
				\mu_{v}(\mathcal{A}(m_{v}))
			=
				\mu_{v}(\Lie(\mathcal{A})(m_{v}))
			=
				\frac{
					\mu_{v}(\Lie(\mathcal{A})(\Order_{v}))
				}{
					\# \Lie(\mathcal{A})(k(v))
				}
			=
				\frac{
					\mu_{v}(\Lie(\mathcal{A})(\Order_{v}))
				}{
					N(v)^{d}
				}.
		\]
	Combining all the above, we get
		\[
				\mu_{v}(A(K_{v}))
			=
					\# \pi_{0}(\mathcal{A}_{v})(k(v))
				\cdot
					\mu_{v}(\Lie(\mathcal{A})(\Order_{v}))
				\cdot
					\frac{
						\# \mathcal{A}^{0}(k(v))
					}{
						N(v)^{d}
					}.
		\]
	The third factor in the right-hand side is $P_{v}(N(v)^{-1})$
	by Prop.\ \ref{prop: special value of Euler factor}.
	Taking the product over $v \not\in U$, we get the result.
\end{proof}

\begin{Prop} \label{prop: volume term in terms of Euler char}
	\[
			\vol \Bigl(
				\prod_{v \not\in U} A(K_{v})
			\Bigr)
		=
				c(A)
			\cdot
				q^{\chi(S, \Lie(\mathcal{A}))}
			\cdot
				\prod_{v \not\in U}
					P_{v}(N(v)^{-1}).
	\]
\end{Prop}

\begin{proof}
	This follows from Prop.\ \ref{prop: measure of Lie ideles and Euler char} and \ref{prop: measure of local points}
\end{proof}

\begin{Prop}
	\[
			\lim_{s \to 1}
				\frac{
					L(A, s)
				}{
					L_{\Sigma}^{\ast}(A, s)
				}
		=
				c(A)
			\cdot
				q^{\chi(S, \Lie(\mathcal{A}))}.
	\]
\end{Prop}

\begin{proof}
	We have
		\begin{align*}
					\frac{L(U, A, s)}{L_{\Sigma}^{\ast}(A, s)}
			&	=
					\frac{
						\prod_{v \not\in U}
							\mu_{v}(A(K_{v}))
					}{
						\mu(\adele_{K} / K)^{d}
					}
			\\
			&	=
					\frac{
						\prod_{v \not\in U}
							\mu_{v}(A(K_{v}))
					}{
						\mu \bigl(
							\Lie(A)(\adele_{K}) / \Lie(A)(K)
						\bigr)
					}
			\\
			&	=
					\vol \Bigl(
						\prod_{v \not\in U} A(K_{v})
					\Bigr)
			\\
			&	=
						c(A)
					\cdot
						q^{\chi(S, \Lie(\mathcal{A}))}
					\cdot
						\prod_{v \not\in U}
							P_{v}(N(v)^{-1}),
		\end{align*}
	where the first and third equalities are by definition,
	the second by Prop.\ \ref{prop: comparison of measure terms of Tate and Kato Trihan}
	and the fourth by Prop.\ \ref{prop: volume term in terms of Euler char}.
	On the other hand,
		\[
				\frac{L(A, s)}{L(U, A, s)}
			=
				\prod_{v \not\in U}
					P_{v}(N(v)^{-s})^{-1}
			\to
				\prod_{v \not\in U}
					P_{v}(N(v)^{-1})^{-1}
		\]
	as $s \to 1$.
	Multiplying these two, we get the result.
\end{proof}

\begin{Cor} \label{cor: BSD with bad Euler factors}
	The formula in Conj.\ \ref{conj: Tate BSD}
	or \ref{conj: Kato Trihan BSD} is equivalent to the formula
		\[
				\lim_{s \to 1}
					\frac{
						L(A, s)
					}{
						(1 - q^{1 - s})^{r}
					}
			=
					\frac{
						\# \Sha(A)
					}{
						\# A(K)_{\tor} \cdot \# B(K)_{\tor}
					}
				\cdot
					\frac{
						\Disc(h)
					}{
						(\log q)^{r}
					}
				\cdot
					c(A)
				\cdot
					q^{\chi(S, \Lie(\mathcal{A}))}.
		\]
\end{Cor}

By \eqref{eq: Riemann Roch},
we have
	\[
			\chi(S, \Lie(\mathcal{A}))
		=
				d \cdot \chi(S, \Order_{S})
			+
				\deg(\det(\Lie(\mathcal{A})))
		=
			d (1 - g) - \deg(\omega),
	\]
where $g$ is the genus of the curve $S$
and $\deg(\omega) = \sum_{v} [k(v) : \F_{q}] v(\omega)$.
Hence we can also write the above conjectural formula as
	\[
			\lim_{s \to 1}
				\frac{
					L(A, s)
				}{
					(s - 1)^{r}
				}
		=
				\frac{
					\# \Sha(A) \cdot \Disc(h)
				}{
					\# A(K)_{\tor} \cdot \# B(K)_{\tor}
				}
			\cdot
				q^{- \deg(\omega) + d (1 - g)}
			\cdot
				c(A).
	\]
If $A$ has good reduction everywhere,
this is the formula in \cite[Thm.\ 4.7]{Bau92},
which omits the factor $c(A)$ since $c(A) = 1$ for such $A$.

In the rest of this paper,
we will rewrite the right-hand side of the formula in Cor.\ \ref{cor: BSD with bad Euler factors}
using Weil-\'etale cohomology of $S$ with coefficients in $\mathcal{A}$.
We begin with the definition of Weil-\'etale cohomology and need some preparations.


\section{Review of Weil-\'etale cohomology}
\label{sec: Review of Weil etale cohomology}

We recall the definition of Weil-\'etale cohomology following \cite{Gei04}.
For a scheme $X$ over $\F_{q}$,
its base change to $\closure{\F_{q}}$ is denoted by $\closure{X}$.
For a sheaf $\mathcal{F}$ on $S_{\et}$,
its pullback to $\closure{S}_{\et}$ is denoted by $\closure{\mathcal{F}}$.
If $\mathcal{F}$ is representable by a scheme locally of finite type over $S$,
then these two pieces of notation are compatible
by a limit argument (\cite[II, Lem.\ 3.3, also Rmk.\ 3.4]{Mil80}).

Let $G \cong \Z$ be the Weil group of $\F_{q}$
and $\phi \in G$ the $q$-th power arithmetic Frobenius.
We denote the category of abelian groups (resp.\ $G$-modules) by $\Ab$ (resp.\ $\Mod_{G}$)
and the category of sheaves of abelian groups on $S_{\et}$ by $\Ab(S_{\et})$.
Consider the left exact functor $\Ab(S_{\et}) \to \Mod_{G}$
sending a sheaf $\mathcal{F}$
to the abelian group $\Gamma(\closure{S}, \closure{\mathcal{F}})$ with its natural $G$-action.
If $\mathcal{F}$ is an injective sheaf,
then $H_{\et}^{i}(\closure{S}, \closure{\mathcal{F}}) = 0$ for $i > 0$
by a limit argument (\cite[III, Lem.\ 1.16]{Mil80}).
Therefore the $i$-th right derived functor of
$\mathcal{F} \mapsto \Gamma(\closure{S}, \closure{\mathcal{F}})$
is $\mathcal{F} \mapsto H_{\et}^{i}(\closure{S}, \closure{\mathcal{F}})$ with the natural $G$-action.
Hence this derived functor agrees with
what is denoted by $R^{i} \Gamma_{\closure{S}}(\gamma^{\ast} \mathcal{F})$
in the notation of \cite[\S 6]{Gei04} by \cite[Lem.\ 6.1]{Gei04}.
Let
	\[
			D^{+}(S_{\et}) \to D^{+}(\Mod_{G}),
		\quad
			\mathcal{F}^{\cdot} \mapsto R \Gamma_{\et}(\closure{S}, \closure{\mathcal{F}^{\cdot}})
	\]
be the total right derived functor
on the bounded below derived categories.
By composing it with the group cohomology functor $R \Gamma(G, \;\cdot\;)$,
we have a triangulated functor to $D^{+}(\Ab)$,
which agrees with what is denoted by $R \Gamma_{W}(S, \gamma^{\ast} \;\cdot\;)$
in the notation of \cite[\S 6]{Gei04}.
Omitting $\gamma^{\ast}$ from the notation, we denote the resulting functor by
	\[
			D^{+}(S_{\et}) \to D^{+}(\Ab),
		\quad
				\mathcal{F}^{\cdot}
			\mapsto
				R \Gamma_{W}(S, \mathcal{F}^{\cdot})
			:=
				R \Gamma \bigl(
					G, R \Gamma_{\et}(\closure{S}, \closure{\mathcal{F}^{\cdot}})
				\bigr).
	\]
One may take this as the definition of Weil-\'etale cohomology of \'etale sheaves,
but see \cite{Gei04} for the full details.

For $\mathcal{F} \in \Ab(S_{\et})$,
we have $H_{W}^{0}(S, \mathcal{F}) = \Gamma(S, \mathcal{F})$.
Since $G \cong \Z$ is generated by $\phi$,
we have a long exact sequence
	\[
			\cdots
		\to
			H_{W}^{i}(S, \mathcal{F})
		\to
			H_{\et}^{i}(\closure{S}, \closure{\mathcal{F}})
		\overset{\phi - 1}{\to}
			H_{\et}^{i}(\closure{S}, \closure{\mathcal{F}})
		\to
			H_{W}^{i + 1}(S, \mathcal{F})
		\to
			\cdots
	\]
and a short exact sequence
	\begin{equation} \label{eq: Frobenius inv and coinv and Weil etale cohomology}
			0
		\to
			H^{i - 1}_{\et}(\closure{S}, \closure{\mathcal{F}})_{G}
		\to
			H_{W}^{i}(S, \mathcal{F})
		\to
			H_{\et}^{i}(\closure{S}, \closure{\mathcal{F}})^{G}
		\to
			0
	\end{equation}
for $\mathcal{F} \in \Ab(S_{\et})$,
where $(\;\cdot\;)_{G}$ and $(\;\cdot\;)^{G}$ denote
the $G$-coinvariants and $G$-invariants, respectively.
By \cite[Cor.\ 5.2]{Gei04}, there exists a canonical long exact sequence
	\begin{equation} \label{eq: etale and Weil etale exact sequence}
			\cdots
		\to
			H_{\et}^{i}(S, \mathcal{F})
		\to
			H_{W}^{i}(S, \mathcal{F})
		\to
			H_{\et}^{i - 1}(S, \mathcal{F}) \tensor \Q
		\to
			H_{\et}^{i + 1}(S, \mathcal{F})
		\to
			\cdots.
	\end{equation}

Let $e \in H^{1}(G, \Z) = \Hom(G, \Z)$ be the homomorphism
sending $\phi$ to $1$.
The cup product with $e$ gives a canonical homomorphism
$e \colon H_{W}^{i}(S, \mathcal{F}) \to H_{W}^{i + 1}(S, \mathcal{F})$.
This agrees with the composite
	\begin{equation} \label{eq: cup with e by geometric cohom}
			H_{W}^{i}(S, \mathcal{F})
		\to
			H_{\et}^{i}(\closure{S}, \closure{\mathcal{F}})^{G}
		\overset{\mathrm{can}}{\to}
			H_{\et}^{i}(\closure{S}, \closure{\mathcal{F}})_{G}
		\to
			H_{W}^{i + 1}(S, \mathcal{F})
	\end{equation}
by \cite[Lem.\ 6.2 b)]{Gei04}.
Since $e \cup e = 0$, we obtain a complex
	\[
			(H_{W}^{\ast}(S, \mathcal{F}), e)
		=
			[
					\cdots
				\overset{e}{\to}
					H_{W}^{i}(S, \mathcal{F})
				\overset{e}{\to}
					H_{W}^{i + 1}(S, \mathcal{F})
				\overset{e}{\to}
					H_{W}^{i + 2}(S, \mathcal{F})
				\overset{e}{\to}
					\cdots
			]
	\]
of abelian groups.


\section{Finite generation for N\'eron model coefficients}
\label{sec: Finite generation for Neron model coefficients}

The N\'eron model $\mathcal{A}$ and its subgroup scheme $\mathcal{A}^{0}$ represent sheaves on $S_{\et}$,
so that their Weil-\'etale cohomology groups
$H_{W}^{\ast}(S, \mathcal{A})$ and $H_{W}^{\ast}(S, \mathcal{A}^{0})$ make sense.
In this section, we study finiteness properties
of $H_{W}^{\ast}(S, \mathcal{A})$ and $H_{W}^{\ast}(S, \mathcal{A}^{0})$.
This is a continuation of what is studied in \cite[Prop.\ 4.2.10]{Suz19} and the paragraph after.

First recall from \cite{Suz19} the commutative group schemes
$\mathbf{H}^{n}(S, \mathcal{A})$ over $\F_{q}$ for each $n$,
a canonical subgroup scheme $\mathbf{H}^{1}(S, \mathcal{A})_{\divis}$ of $\mathbf{H}^{1}(S, \mathcal{A})$
and similar objects $\mathbf{H}^{n}(S, \mathcal{A}^{0})$, $\mathbf{H}^{1}(S, \mathcal{A}^{0})_{\divis}$.
We use the following results.

\begin{Prop}[\cite{Suz19}] \label{prop: cohom group schemes} \mbox{}
	\begin{enumerate}
		\item \label{item: cohom group schemes: rational points}
			The group of $\closure{\F_{q}}$-points of $\mathbf{H}^{n}(S, \mathcal{A})$ is given by
			$H_{\et}^{n}(\closure{S}, \closure{\mathcal{A}})$
			including the $G$-actions
			(\cite[Prop.\ 2.7.8]{Suz19}).
		\item \label{item: cohom group schemes: basic structure and cohom dim}
			$\mathbf{H}^{n}(S, \mathcal{A})$ is the perfection (inverse limit along Frobenius morphisms)
			of a smooth group scheme over $\F_{q}$ for any $n$ and
			$\mathbf{H}^{n}(S, \mathcal{A}) = 0$ for $n \ne 0, 1, 2$
			(\cite[Thm.\ 3.4.1 (1)]{Suz19}).
		\item \label{item: cohom group schemes: H zero}
			The identity component of $\mathbf{H}^{0}(S, \mathcal{A})$ is the perfection of an abelian variety
			and the component group of $\mathbf{H}^{0}(S, \mathcal{A})$ is
			an \'etale group with finitely generated group of geometric points
			(\cite[Thm.\ 3.4.1 (2)]{Suz19}).
		\item \label{item: cohom group schemes: H two}
			$\mathbf{H}^{2}(S, \mathcal{A})$ is a torsion \'etale group
			whose Pontryagin dual is the profinite Tate module of an abelian variety
			(\cite[Thm.\ 3.4.1 (2), (6a)]{Suz19}).
		\item \label{item: cohom group schemes: H one modulo div}
			The quotient $\mathbf{H}^{1}(S, \mathcal{A}) / \divis$
			of $\mathbf{H}^{1}(S, \mathcal{A})$ by $\mathbf{H}^{1}(S, \mathcal{A})_{\divis}$
			is the perfection of a commutative algebraic group with unipotent identity component
			(\cite[Thm.\ 3.4.1 (2)]{Suz19}).
		\item \label{item: cohom group schemes: H one div}
			$\mathbf{H}^{1}(S, \mathcal{A})_{\divis}$ is a divisible torsion \'etale group scheme
			with finite $n$-torsion part for any $n \ge 1$
			(\cite[Thm.\ 3.4.1 (2)]{Suz19}).
		\item \label{item: cohom group schemes: finiteness of Sha}
			Let $T(\mathbf{H}^{1}(S, \mathcal{A})_{\divis})$ be the profinite Tate module of
			$\mathbf{H}^{1}(S, \mathcal{A})_{\divis}$.
			Let $V(\mathbf{H}^{1}(S, \mathcal{A})_{\divis})$ be
			$T(\mathbf{H}^{1}(S, \mathcal{A})_{\divis}) \tensor \Q$.
			Then
				\[
						(V(\mathbf{H}^{1}(S, \mathcal{A})_{\divis}))^{G}
					=
						(V(\mathbf{H}^{1}(S, \mathcal{A})_{\divis}))_{G}
					=
						0
				\]
			if and only if $\Sha(A)$ is finite
			(\cite[Prop.\ 4.2.5]{Suz19}).
		\item \label{item: cohom group schemes: for connected Neron}
			The statements above also hold with $\mathcal{A}$ replaced by $\mathcal{A}^{0}$
			(\cite[Thm.\ 3.4.1 (3), Prop.\ 3.2.4]{Suz19}).
	\end{enumerate}
\end{Prop}

In \cite[\S 4]{Suz19}, the $G$-coinvariants
$(V(\mathbf{H}^{1}(S, \mathcal{A})_{\divis}))_{G}$
is taken in (a category containing) the ind-category of profinite abelian groups
(see also the proof of Prop.\ \ref{prop: finite generation} below).
The object $(V(\mathbf{H}^{1}(S, \mathcal{A})_{\divis}))_{G}$ is zero
as an ind-object of profinite abelian groups
if and only if it is zero as an (abstract) abelian group,
since the $l$-adic Tate module $T_{l}(\mathbf{H}^{1}(S, \mathcal{A})_{\divis})$
is a finite free $\Z_{l}$-module for any prime $l$
by Assertion \eqref{item: cohom group schemes: H one div}.
Hence one may equivalently take the $G$-coinvariants
$(V(\mathbf{H}^{1}(S, \mathcal{A})_{\divis}))_{G}$
in the category of abelian groups
in Assertion \eqref{item: cohom group schemes: finiteness of Sha}.
A priori, $(V(\mathbf{H}^{1}(S, \mathcal{A})_{\divis}))_{G}$ might contain
a subgroup isomorphic to $(\prod_{l} \Z / l \Z) / (\bigoplus_{l} \Z / l \Z)$ for example.

We denote the groups of $\closure{\F_{q}}$-points of
$\mathbf{H}^{1}(S, \mathcal{A})_{\divis}$ and $\mathbf{H}^{1}(S, \mathcal{A}) / \divis$
by $H_{\et}^{1}(\closure{S}, \closure{\mathcal{A}})_{\divis}$
and $H_{\et}^{1}(\closure{S}, \closure{\mathcal{A}}) / \divis$, respectively.
We use the same notation with $\mathcal{A}$ replaced by $\mathcal{A}^{0}$.

\begin{Prop} \label{prop: geometric cohomology} \mbox{}
	\begin{enumerate}
		\item \label{item: geometric cohomology: cohom dimension}
			We have $H_{\et}^{n}(\closure{S}, \closure{\mathcal{A}}) = 0$
			for $n \ne 0, 1, 2$.
		\item \label{item: geometric cohomology: H zero}
			The groups $\closure{\mathcal{A}}(\closure{S})^{G}$
			and $\closure{\mathcal{A}}(\closure{S})_{G}$
			are finitely generated.
		\item \label{item: geometric cohomology: H two}
			The group $H_{\et}^{2}(\closure{S}, \closure{\mathcal{A}})^{G}$ is finite,
			and $H_{\et}^{2}(\closure{S}, \closure{\mathcal{A}})_{G}$ is trivial.
		\item \label{item: geometric cohomology: H one mod div}
			The groups $(H_{\et}^{1}(\closure{S}, \closure{\mathcal{A}}) / \divis)^{G}$
			and $(H_{\et}^{1}(\closure{S}, \closure{\mathcal{A}}) / \divis)_{G}$ are finite.
		\item \label{item: geometric cohomology: H one div is torsion cofinite}
			The group $H_{\et}^{1}(\closure{S}, \closure{\mathcal{A}})_{\divis}$ is divisible torsion
			with finite $n$-torsion part for any $n \ge 1$.
		\item \label{item: geometric cohomology: H one div and Sha}
			Let $T(H_{\et}^{1}(\closure{S}, \closure{\mathcal{A}})_{\divis})$ be
			the profinite Tate module of $H_{\et}^{1}(\closure{S}, \closure{\mathcal{A}})_{\divis}$.
			Let $V(H_{\et}^{1}(\closure{S}, \closure{\mathcal{A}})_{\divis})$ be
			$T(H_{\et}^{1}(\closure{S}, \closure{\mathcal{A}})_{\divis}) \tensor \Q$.
			Then we have
				\[
						(V(H_{\et}^{1}(\closure{S}, \closure{\mathcal{A}})_{\divis}))^{G}
					=
						(V(H_{\et}^{1}(\closure{S}, \closure{\mathcal{A}})_{\divis}))_{G}
					=
						0
				\]
			if and only if $\Sha(A)$ is finite.
		\item \label{item: geometric cohomology: connected Neron instead of Neron}
			The statements above also hold
			with $\mathcal{A}$ replaced by $\mathcal{A}^{0}$.
	\end{enumerate}
\end{Prop}

\begin{proof}
	\eqref{item: geometric cohomology: cohom dimension}
	This follows from Prop.\ \ref{prop: cohom group schemes}
	\eqref{item: cohom group schemes: rational points},
	\eqref{item: cohom group schemes: basic structure and cohom dim}.
	
	\eqref{item: geometric cohomology: H zero}
	First, the endomorphism $\phi - 1$ on any commutative connected algebraic group over $\F_{q}$
	is surjective with finite kernel by Lang's theorem.
	The same is true with ``commutative connected algebraic group''
	replaced by the perfection of such a group.
	Hence Prop.\ \ref{prop: cohom group schemes}
	\eqref{item: cohom group schemes: H zero} implies the result.
	
	\eqref{item: geometric cohomology: H two}, \eqref{item: cohom group schemes: H two}
	The same argument as the proof of the previous assertion applies
	by Prop.\ \ref{prop: cohom group schemes}
	\eqref{item: cohom group schemes: H two},
	\eqref{item: cohom group schemes: H one modulo div},
	respectively.
	
	\eqref{item: geometric cohomology: H one div is torsion cofinite},
	\eqref{item: geometric cohomology: H one div and Sha},
	\eqref{item: geometric cohomology: connected Neron instead of Neron}
	These follow from
	Prop.\ \ref{prop: cohom group schemes}
	\eqref{item: cohom group schemes: H one div},
	\eqref{item: cohom group schemes: finiteness of Sha},
	\eqref{item: cohom group schemes: for connected Neron},
	respectively.
\end{proof}

\begin{Prop} \label{prop: Weil etale cohom} \mbox{}
	\begin{enumerate}
		\item \label{item: Weil etale cohom: cohom dimension}
			We have $H_{W}^{n}(S, \mathcal{A}) = 0$ for $n \ne 0, 1, 2$.
		\item \label{item: Weil etale cohom: H zero finitely generated}
			The group $H_{W}^{0}(S, \mathcal{A})$ is finitely generated.
		\item \label{item: Weil etale cohom: H two torsion}
			The group $H_{W}^{2}(S, \mathcal{A})$ is torsion.
		\item \label{item: Weil etale cohom: H one and geom H one}
			The group $H_{W}^{1}(S, \mathcal{A})$ is finitely generated if and only if
			the torsion group $(H_{\et}^{1}(\closure{S}, \closure{\mathcal{A}})_{\divis})^{G}$ is finite.
		\item \label{item: Weil etale cohom: H two and geome H one}
			The group $H_{W}^{2}(S, \mathcal{A})$ is finite if and only if
			the divisible group $(H_{\et}^{1}(\closure{S}, \closure{\mathcal{A}})_{\divis})_{G}$ is trivial.
		\item \label{item: Weil etale cohom: connected Neron instead of Neron}
			The statements above also hold with $\mathcal{A}$ replaced by $\mathcal{A}^{0}$.
	\end{enumerate}
\end{Prop}

\begin{proof}
	\eqref{item: Weil etale cohom: cohom dimension}
	follows from Prop.\ \ref{prop: geometric cohomology}
	\eqref{item: geometric cohomology: cohom dimension}
	and \eqref{item: geometric cohomology: H two}
	and the exact sequence \eqref{eq: Frobenius inv and coinv and Weil etale cohomology}.
	The rest of the statements follow from the exact sequences
			\begin{gather*}
				0
			\to
				\closure{\mathcal{A}}(\closure{S})_{G}
			\to
				H_{W}^{1}(S, \mathcal{A})
			\to
				H_{\et}^{1}(\closure{S}, \closure{\mathcal{A}})^{G}
			\to
				0,
			\\
				0
			\to
				H_{\et}^{1}(\closure{S}, \closure{\mathcal{A}})_{G}
			\to
				H_{W}^{2}(S, \mathcal{A})
			\to
				H_{\et}^{2}(\closure{S}, \closure{\mathcal{A}})^{G}
			\to
				0,
		\end{gather*}
	and
		\begin{align*}
					0
			&	\to
					(H_{\et}^{1}(\closure{S}, \closure{\mathcal{A}})_{\divis})^{G}
				\to
					H_{\et}^{1}(\closure{S}, \closure{\mathcal{A}})^{G}
				\to
					(H_{\et}^{1}(\closure{S}, \closure{\mathcal{A}}) / \divis)^{G}
			\\
			&	\to
					(H_{\et}^{1}(\closure{S}, \closure{\mathcal{A}})_{\divis})_{G}
				\to
					H_{\et}^{1}(\closure{S}, \closure{\mathcal{A}})_{G}
				\to
					(H_{\et}^{1}(\closure{S}, \closure{\mathcal{A}}) / \divis)_{G}
				\to
					0,
		\end{align*}
	and Prop.\ \ref{prop: geometric cohomology}.
\end{proof}

Of course $H_{W}^{0}(S, \mathcal{A}) \cong A(K)$ is finitely generated
also by the Mordell-Weil theorem.
The group $H_{W}^{0}(S, \mathcal{A}^{0}) \cong \mathcal{A}^{0}(S)$ is
a finite index subgroup of $A(K)$.

\begin{Prop} \label{prop: finite generation}
	The group $H_{W}^{1}(S, \mathcal{A})$ is finitely generated
	if and only if all the groups $H_{W}^{\ast}(S, \mathcal{A})$ are finitely generated
	if and only if $\Sha(A)$ is finite.
	The same is true with $\mathcal{A}$ replaced by $\mathcal{A}^{0}$.
\end{Prop}

\begin{proof}
	Let $\mathcal{C}$ be the category of finite abelian groups.
	Let $\Pro(\mathcal{C})$ be the pro-category of $\mathcal{C}$
	and $\Ind(\Pro(\mathcal{C}))$ the ind-category of $\Pro(\mathcal{C})$
	(\cite[Def.\ 6.1.1]{KS06}).
	They are abelian categories by \cite[Thm.\ 8.6.5 (i)]{KS06}.
	The category $\Pro(\mathcal{C})$ is just the category of profinite abelian groups.
	Since the natural functor $\mathcal{C} \to \Pro(\mathcal{C})$ is fully faithful,
	the induced functor $\Ind(\mathcal{C}) \to \Ind(\Pro(\mathcal{C}))$
	is also fully faithful
	(\cite[Prop.\ 6.1.10]{KS06}),
	where the ind-category $\Ind(\mathcal{C})$ of $\mathcal{C}$
	is just the category of torsion abelian groups.
	
	Consider the short exact sequence
		\[
				0
			\to
				T(H_{\et}^{1}(\closure{S}, \closure{A})_{\divis})
			\to
				V(H_{\et}^{1}(\closure{S}, \closure{A})_{\divis})
			\to
				H_{\et}^{1}(\closure{S}, \closure{A})_{\divis}
			\to
				0
		\]
	of $G$-modules.
	We view $H_{\et}^{1}(\closure{S}, \closure{A})_{\divis} \in \Ind(\mathcal{C})$,
	$T(H_{\et}^{1}(\closure{S}, \closure{A})_{\divis}) \in \Pro(\mathcal{C})$ and
	$V(H_{\et}^{1}(\closure{S}, \closure{A})_{\divis}) \in \Ind(\Pro(\mathcal{C}))$.
	(The object $V(H_{\et}^{1}(\closure{S}, \closure{A})_{\divis})$ is of course a locally compact group.)
	Consequently, we may view the above sequence as a short exact sequence of
	$G$-module objects in the abelian category $\Ind(\Pro(\mathcal{C}))$.
	We have the induced long exact sequence
		\begin{align*}
					0
			&	\to
					(T(H_{\et}^{1}(\closure{S}, \closure{A})_{\divis}))^{G}
				\to
					(V(H_{\et}^{1}(\closure{S}, \closure{A})_{\divis}))^{G}
				\to
					(H_{\et}^{1}(\closure{S}, \closure{A})_{\divis})^{G}
			\\
			&	\to
					(T(H_{\et}^{1}(\closure{S}, \closure{A})_{\divis}))_{G}
				\to
					(V(H_{\et}^{1}(\closure{S}, \closure{A})_{\divis}))_{G}
				\to
					(H_{\et}^{1}(\closure{S}, \closure{A})_{\divis})_{G}
				\to
					0
		\end{align*}
	in $\Ind(\Pro(\mathcal{C}))$.
	
	Therefore if $\Sha(A)$ is finite,
	then by Prop.\ \ref{prop: geometric cohomology}
	\eqref{item: geometric cohomology: H one div and Sha},
	the groups $(T(H_{\et}^{1}(\closure{S}, \closure{A})_{\divis}))^{G}$
	and $(H_{\et}^{1}(\closure{S}, \closure{A})_{\divis})_{G}$ are trivial,
	and we have an isomorphism
	from $(H_{\et}^{1}(\closure{S}, \closure{A})_{\divis})^{G} \in \Ind(\mathcal{C})$
	to $(T(H_{\et}^{1}(\closure{S}, \closure{A})_{\divis}))_{G} \in \Pro(\mathcal{C})$
	in $\Ind(\Pro(\mathcal{C}))$.
	This implies that these isomorphic groups are in $\mathcal{C}$, i.e.\ finite.
	Therefore all of $H_{W}^{\ast}(S, \mathcal{A})$ are finitely generated
	by Prop.\ \ref{prop: geometric cohomology}.
	
	Conversely, if $H_{W}^{1}(S, \mathcal{A})$ is finitely generated,
	then $(H_{\et}^{1}(\closure{S}, \closure{\mathcal{A}})_{\divis})^{G}$ is finite.
	Therefore by Prop.\ \ref{prop: geometric cohomology}
	\eqref{item: geometric cohomology: H one div is torsion cofinite},
	the endomorphism $\phi - 1$ on the $l$-primary part of
	$H_{\et}^{1}(\closure{S}, \closure{\mathcal{A}})_{\divis}$
	is surjective for any prime number $l$ (possibly equal to $p$)
	and invertible for almost all $l$.
	Hence $\phi - 1$ on $T(H_{\et}^{1}(\closure{S}, \closure{A})_{\divis})$
	is injective with finite cokernel.
	Hence $\phi - 1$ on $V(H_{\et}^{1}(\closure{S}, \closure{A})_{\divis})$
	is an isomorphism.
	Thus by Prop.\ \ref{prop: geometric cohomology}
	\eqref{item: geometric cohomology: H one div and Sha},
	we know that $\Sha(A)$ is finite.
	
	The case of $\mathcal{A}^{0}$ can be treated similarly
	(or reduced to the case of $\mathcal{A}$).
\end{proof}


\section{Duality}
\label{sec: Duality}

We recall the duality result \cite[Prop.\ 4.2.10]{Suz19} on $H_{W}^{\ast}(S, \mathcal{A})$.
The Poincar\'e bundle $P$ on $A \times_{K} B$ canonically extends to
a line bundle $\mathcal{P}$ on $\mathcal{A} \times_{S} \mathcal{B}^{0}$ and
defines a morphism $\mathcal{A} \tensor^{L} \mathcal{B}^{0} \to \Gm[1]$ in $D(S_{\et})$
by \cite[IX, 1.4.3]{Gro72},
where $\tensor^{L}$ denotes the derived tensor product.
See \cite[III, Appendix C]{Mil06} for a good review of this theory.
Applying $R \Gamma_{W}(S, \;\cdot\;)$ and cup product, we have a morphism
	\begin{equation} \label{eq: cup product pairing}
				R \Gamma_{W}(S, \mathcal{A})
			\tensor^{L}
				R \Gamma_{W}(S, \mathcal{B}^{0})
		\to
			R \Gamma_{W}(S, \Gm)[1]
	\end{equation}
in $D(\Ab)$.
By \cite[Prop.\ 7.4]{Gei04}, we have canonical isomorphisms
	\begin{equation} \label{eq: Weil etal cohom of Gm}
			H_{W}^{n}(S, \Gm)
		\cong
			\begin{cases}
				\F_{q}^{\times} & \text{if } n = 0, \\
				\Pic(S) & \text{if } n = 1, \\
				\Z & \text{if } n = 2, \\
				0 & \text{if } n \ge 3.
			\end{cases}
	\end{equation}
In particular, we have a canonical morphism
$R \Gamma_{W}(S, \Gm) \to \Z[-2]$.
This induces a morphism
	\begin{equation} \label{eq: derived duality pairing}
				R \Gamma_{W}(S, \mathcal{A})
			\tensor^{L}
				R \Gamma_{W}(S, \mathcal{B}^{0})
		\to
			\Z[-1].
	\end{equation}

\begin{Prop} \label{prop: duality}
	Assume that $\Sha(A)$ is finite
	(which implies finite generation of
	$H_{W}^{\ast}(S, \mathcal{A})$ and $H_{W}^{\ast}(S, \mathcal{B}^{0})$
	by Prop.\ \ref{prop: finite generation}).
	Then the morphism
		\[
				R \Gamma_{W}(S, \mathcal{A})
			\to
				R \Hom \bigl(
					R \Gamma_{W}(S, \mathcal{B}^{0}),
					\Z
				\bigr)[-1]
		\]
	induced by \eqref{eq: derived duality pairing} is an isomorphism in $D(\Ab)$.
	In particular, for any $n$, we have a perfect pairing
		\[
					H_{W}^{n}(S, \mathcal{A}) / \tor
				\times
					H_{W}^{1 - n}(S, \mathcal{B}^{0}) / \tor
			\to
				\Z
		\]
	of finite free abelian groups and a perfect pairing
		\[
					H_{W}^{n}(S, \mathcal{A})_{\tor}
				\times
					H_{W}^{2 - n}(S, \mathcal{B}^{0})_{\tor}
			\to
				\Q / \Z
		\]
	of finite abelian groups.
\end{Prop}

\begin{proof}
	This follows from \cite[Prop.\ 4.2.10]{Suz19}.
\end{proof}


\section{Euler characteristics for N\'eron models}
\label{sec: Euler characteristics for Neron models}

In this section, we assume that
	\[
		\# \Sha(A) < \infty,
	\]
so that the Weil-\'etale cohomology groups
$H_{W}^{\ast}(S, \mathcal{A})$ are finitely generated
by Prop.\ \ref{prop: finite generation}.
We relate $\chi(H_{W}^{\ast}(S, \mathcal{A}, e))$ to the product
	\[
			\frac{
				\# \Sha(A)
			}{
				\# A(K)_{\tor} \cdot \# B(K)_{\tor}
			}
		\cdot
			\frac{
				\Disc(h)
			}{
				(\log q)^{r}
			}
		\cdot
			c(A)
	\]
that appears in Cor.\ \ref{cor: BSD with bad Euler factors},
thereby finishing the proof of Thm.\ \ref{thm: BSD by Weil etale}.

As in \S \ref{sec: Review of Weil etale cohomology},
the cup product with $e \in H^{1}(G, \Z)$ turns the groups $H^{\ast}_{W}(S, \mathcal{A})$ into a complex
$(H_{W}^{\ast}(S, \mathcal{A}), e)$.
The rationalized complex $(H^{\ast}_{W}(S, \mathcal{A}) \tensor \Q, e)$ is exact
by the general result on uniquely divisible sheaves \cite[Cor.\ 5.2]{Gei04}.
Hence the cohomology groups of the complex $(H^{\ast}_{W}(S, \mathcal{A}), e)$ are finite.
Its Euler characteristic $\chi(H_{W}^{\ast}(S, \mathcal{A}), e)$ is thus well-defined.
On the other hand, the groups $H_{W}^{\ast}(S, \mathcal{A})_{\tor}$ are finite,
so that $\chi(H_{W}^{\ast}(S, \mathcal{A})_{\tor})$ also is well-defined.

\begin{Prop} \label{prop: Euler char of torsion part of Weil etale cohom}
	\[
			\chi(H_{W}^{\ast}(S, \mathcal{A})_{\tor})^{-1}
		=
				\frac{
					\# \Sha(A)
				}{
					\# A(K)_{\tor} \cdot \# B(K)_{\tor}
				}
			\cdot
				\frac{
					c(A)
				}{
					[B(K) / \tor : \mathcal{B}^{0}(S) / \tor]
				}.
	\]
\end{Prop}

\begin{proof}
	We have $H_{W}^{0}(S, \mathcal{A}) = A(K)$.
	Also the finite group $H_{W}^{2}(S, \mathcal{A})$ is Pontryagin dual to $\mathcal{B}^{0}(S)_{\tor}$
	by Prop.\ \ref{prop: duality}.
	
	We treat $H_{W}^{1}(S, \mathcal{A})$.
	By Prop.\ \ref{prop: duality},
	we have $\# H_{W}^{1}(S, \mathcal{A})_{\tor} = \# H_{W}^{1}(S, \mathcal{B}^{0})_{\tor}$.
	By \eqref{eq: etale and Weil etale exact sequence},
	we have a natural exact sequence
		\[
				0
			\to
				H_{\et}^{1}(S, \mathcal{B}^{0})
			\to
				H_{W}^{1}(S, \mathcal{B}^{0})
			\to
				\mathcal{B}^{0}(S) \tensor \Q.
		\]
	Hence $H_{\et}^{1}(S, \mathcal{B}^{0})_{\tor} \isomto H_{W}^{1}(S, \mathcal{B}^{0})_{\tor}$.
	By \cite[III, Prop.\ 9.2]{Mil06},
	we have a natural exact sequence
		\[
				0
			\to
				\mathcal{B}^{0}(S)
			\to
				B(K)
			\to
				\bigoplus_{v}
					\pi_{0}(\mathcal{B}_{v})(k(v))
			\to
				H_{\et}^{1}(S, \mathcal{B}^{0})
			\to
				\Sha(B)
			\to
				0.
		\]
	Hence $H_{\et}^{1}(S, \mathcal{B}^{0})$ is finite and
		\[
				\# H_{\et}^{1}(S, \mathcal{B}^{0})
			=
				\frac{
						\# \Sha(B)
					\cdot
						c(B)
				}{
					[B(K) : \mathcal{B}^{0}(S)]
				}.
		\]
	We have $\# \Sha(B) = \# \Sha(A)$
	by the perfectness of the Cassels-Tate pairing \cite[III, Cor.\ 9.5]{Mil06}
	and $c(B) = c(A)$ by the perfectness of the Grothendieck pairing \cite[III, Thm.\ 7.11]{Mil06}.
	Thus
		\[
				\# H_{W}^{1}(S, \mathcal{A})_{\tor}
			=
					\frac{
						\# \Sha(A) \cdot c(A)
					}{
						[B(K) : \mathcal{B}^{0}(S)]
					}.
		\]
	
	Therefore
		\begin{align*}
					\chi(H_{W}^{\ast}(S, \mathcal{A})_{\tor})^{-1}
			&	=
					\frac{
						\# \Sha(A)
					}{
						\# A(K)_{\tor} \cdot \# \mathcal{B}^{0}(S)_{\tor}
					}
				\cdot
					\frac{
						c(A)
					}{
						[B(K) : \mathcal{B}^{0}(S)]
					}
			\\
			&	=
					\frac{
						\# \Sha(A)
					}{
						\# A(K)_{\tor} \cdot \# B(K)_{\tor}
					}
				\cdot
					\frac{
						c(A)
					}{
						[B(K) / \tor : \mathcal{B}^{0}(S) / \tor]
					}.
		\end{align*}
\end{proof}

As in \S \ref{sec: Duality},
let $P$ be the Poincar\'e bundle on $A \times_{K} B$
and $\mathcal{P}$ its canonical extension to $\mathcal{A} \times_{S} \mathcal{B}^{0}$.
By pullback, we have a pairing on
$\mathcal{A}(S) \times \mathcal{B}^{0}(S) = A(K) \times \mathcal{B}^{0}(S)$
with values in $\Pic(S)$.
\newcommand{\cuppair}{\langle, \rangle}
Let $\pairing{\;}{\;}$ be the pairing defined by the composite of the maps
	\[
			A(K) \times \mathcal{B}^{0}(S)
		\to
			\Pic(S)
		\to
			\Z,
	\]
where the last map is the degree map.
This pairing $\pairing{\;}{\;}$ is non-degenerate modulo torsion subgroups
by \cite[Lem.\ 9 iii), Satz 11]{Sch82}.

\begin{Prop} \label{prop: Euler char of torsion free quotient of Weil etale cohom}
	\[
			\chi(H_{W}^{\ast}(S, \mathcal{A}) / \tor, e)^{-1}
		=
			\Disc(\pairing{\;}{\;}).
	\]
\end{Prop}

\begin{proof}
	Since $H_{W}^{2}(S, \mathcal{A})$ is torsion
	by Prop.\ \ref{prop: Weil etale cohom} \eqref{item: Weil etale cohom: H two and geome H one},
	the only relevant morphism for the left-hand side is
	$e \colon H_{W}^{0}(S, \mathcal{A}) / \tor \to H_{W}^{1}(S, \mathcal{A}) / \tor$.
	By \eqref{eq: cup with e by geometric cohom}, \eqref{eq: cup product pairing}
	and \eqref{eq: Weil etal cohom of Gm},
	we have a commutative diagram
		\[
			\begin{CD}
					H_{W}^{0}(S, \mathcal{A}) \times H_{W}^{0}(S, \mathcal{B}^{0})
				@>>>
					H_{W}^{1}(S, \Gm)
				@=
					\Pic(S)
				\\
				@VV e \times \id V
				@VV e V
				@VV \deg V
				\\
					H_{W}^{1}(S, \mathcal{A}) \times H_{W}^{0}(S, \mathcal{B}^{0})
				@>>>
					H_{W}^{2}(S, \Gm)
				@=
					\Z.
			\end{CD}
		\]
	The upper pairing gives $\pairing{\;}{\;}$.
	The lower pairing modulo torsion subgroups is perfect
	by Prop.\ \ref{prop: duality}.
	Therefore the homomorphism
	$e \colon H_{W}^{0}(S, \mathcal{A})_{/ \tor} \to H_{W}^{1}(S, \mathcal{A})_{/ \tor}$
	can be identified with the injective homomorphism
	$A(K) / \tor \into \Hom(\mathcal{B}^{0}(S) / \tor, \Z)$ given by $\pairing{\;}{\;}$.
	This implies the result.
\end{proof}

\begin{Prop}
	\[
			\chi(H_{W}^{\ast}(S, \mathcal{A}), e)^{-1}
		=
			\frac{
				\# \Sha(A)
			}{
				\# A(K)_{\tor} \cdot \# B(K)_{\tor}
			}
		\cdot
			\frac{
				\Disc(h)
			}{
				(\log q)^{r}
			}
		\cdot
			c(A).
	\]
\end{Prop}

\begin{proof}
	By the displayed equation right before \cite[Theorem]{Sch82},
	we have
		\[
				\frac{\Disc(h)}{(\log q)^{r}}
			=
				\frac{
					\Disc(\pairing{\;}{\;})
				}{
					[B(K) / \tor : \mathcal{B}^{0}(S) / \tor]
				}.
		\]
	Also we have
		\begin{align*}
					\chi(H_{W}^{\ast}(S, \mathcal{A}), e)
			&	=
						\chi(H_{W}^{\ast}(S, \mathcal{A})_{\tor}, e)
					\cdot
						\chi(H_{W}^{\ast}(S, \mathcal{A}) / \tor, e)
			\\
			&	=
						\chi(H_{W}^{\ast}(S, \mathcal{A})_{\tor})
					\cdot
						\chi(H_{W}^{\ast}(S, \mathcal{A}) / \tor, e);
		\end{align*}
	see the proof of \cite[Thm.\ 9.1]{Gei04}.
	Hence the result follows from the previous two propositions.
\end{proof}

\begin{Prop}
	The formula in Conj.\ \ref{conj: Tate BSD}
	or \ref{conj: Kato Trihan BSD} is equivalent to the formula
		\[
				\lim_{s \to 1}
					\frac{L(A, s)}{(1 - q^{1 - s})^{r}}
			=
				\chi(H_{W}^{\ast}(S, \mathcal{A}), e)^{-1}
				\cdot q^{\chi(S, \Lie(\mathcal{A}))}.
		\]
\end{Prop}

\begin{proof}
	This follows from the previous proposition
	and Cor.\ \ref{cor: BSD with bad Euler factors}
\end{proof}

Now Thm.\ \ref{thm: BSD by Weil etale} is a consequence of this proposition
and the result of Kato-Trihan \cite[Chap.\ I, Thm.]{KT03}.


\section{Integral models for $l$-adic and $p$-adic cohomology}
\label{sec: Integral models for l adic and p adic cohomology}

In this section,
we will see that the Weil-\'etale cohomology $H_{W}^{\ast}(S, \mathcal{A})$ is an integral model
for the corresponding $l$-adic and $p$-adic cohomology theory
if $\Sha(A)$ is finite.
This is a N\'eron model version of the corresponding result \cite[Thm.\ 8.4]{Gei04}
for motivic Tate twists $\Z(n)$.
We follow Jannsen's adic formalism \cite{Jan88}.

We need some notation about inverse limits.
Let $l$ be a prime number that may be equal to $p$.
Let $S_{\fppf}$ be the fppf site of $S$.
Let $\Ab(S_{\fppf})^{\N}$ be the category of inverse systems in $\Ab(S_{\fppf})$
indexed by positive integers with the usual ordering.
It has enough injectives (\cite[(1.1 a)]{Jan88}).
As in \cite[\S 3]{Jan88} (adapted to the fppf site),
the functor $\Ab(S_{\fppf})^{\N} \to \Ab$ given by
$\{\mathcal{F}_{i}\} \mapsto \invlim_{i} \Gamma(S, \mathcal{F}_{i})$
is denoted by $\Gamma(S, \{\mathcal{F}_{i}\})$,
with right derived functors $H_{\fppf}^{n}(S, \{\mathcal{F}_{i}\})$ and
total right derived functor $R \Gamma_{\fppf}(S, \{\mathcal{F}_{i}\})$.
A system $\{\mathcal{F}_{i}\}$ is said to be ML-zero (Mittag-Leffler zero; \cite[(1.10)]{Jan88})
if for any $i \ge 1$, there exists an integer $j = j(i) \ge 1$ such that
the transition morphism $\mathcal{F}_{i + j} \to \mathcal{F}_{i}$ is zero.
In this case, we have $R \Gamma_{\fppf}(S, \{\mathcal{F}_{i}\}) = 0$ (use \cite[(1.11), (3.1)]{Jan88}).
The ML-zero systems form a Serre subcategory of $\Ab(S_{\fppf})^{\N}$ (\cite[(1.12)]{Jan88}).
Two objects of $\Ab(S_{\fppf})^{\N}$ are said to be ML-isomorphic
if they are isomorphic in the quotient category of $\Ab(S_{\fppf})^{\N}$ by ML-zero systems.

As in \cite[(5.1)]{Jan88},
define a functor $\underline{T}_{l} \colon \Ab(S_{\fppf}) \to \Ab(S_{\fppf})^{\N}$
by sending a sheaf $\mathcal{F}$ to the inverse system of sheaves
	\[
			\cdots
		\overset{l}{\to}
			{}_{l^{2}} \mathcal{F}
		\overset{l}{\to}
			{}_{l} \mathcal{F},
	\]
where ${}_{l^{i}} \mathcal{F}$ means the kernel of multiplication by $l^{i}$ on $\mathcal{F}$.
By \cite[(5.1 a)]{Jan88},
for any $\mathcal{F} \in \Ab(S_{\fppf})$,
we have a canonical isomorphism
	\begin{equation} \label{eq: first derived functor of formal Tate module}
			R^{1} \underline{T}_{l} \mathcal{F}
		\cong
			\{ \mathcal{F} \tensor \Z / l^{i} \Z\}_{i},
	\end{equation}
where the transition morphisms are the natural reduction morphisms,
and $R^{n} \underline{T}_{l} \mathcal{F} = 0$ for $n \ge 2$.
As before, let $T_{l} = \Hom(\Q_{l} / \Z_{l}, \;\cdot\;) \colon \Ab \to \Ab$ be
the $l$-adic Tate module functor.
The natural isomorphism
	$
			\Gamma(S, \underline{T}_{l}(\mathcal{F}))
		\cong
			T_{l}(\Gamma(S, \mathcal{F}))
	$
induces a canonical isomorphism
	\[
			R \Gamma_{\fppf}(S, R \underline{T}_{l}(\mathcal{F}))
		\cong
			R T_{l}(R \Gamma_{\fppf}(S, \mathcal{F}))
	\]
in $D(\Ab)$ for $F \in \Ab(S_{\fppf})$
by \cite[(5.2), (5.4)]{Jan88}.
The same definitions and statements hold for the \'etale site $S_{\et}$.
We use similar notation $\underline{T}_{l} \colon \Ab(S_{\et}) \to \Ab(S_{\et})^{\N}$,
$R \Gamma_{\et}(S, \{\mathcal{F}_{i}\})$ etc.\ for the \'etale versions.

For $C^{\cdot} \in D^{+}(\Ab)$,
we have
	\[
		R T_{l}(C^{\cdot})[1] \cong R \Hom(\Q_{l} / \Z_{l}, C^{\cdot})[1].
	\]
We have a natural morphisms
$C^{\cdot} \to R T_{l}(C^{\cdot})[1]$
by applying $R \Hom(\;\cdot\;, C^{\cdot})$ to the morphisms $\Q_{l} / \Z_{l}[-1] \to \Q / \Z[-1] \to \Z$.
This induces a natural morphism
	\[
		C^{\cdot} \tensor \Z_{l} \to R T_{l}(C^{\cdot})[1]
	\]
since $R T_{l}(C^{\cdot})$ is represented by a complex of $\Z_{l}$-modules.
(Note that $\Z_{l}$ is flat and hence the functor $(\;\cdot\;) \tensor \Z_{l}$ is exact
inducing a triangulated functor on the derived categories.)
The above morphism is an isomorphism if $C^{\cdot}$ has finitely generated cohomology groups.
On the other hand, if $C^{\cdot}$ has uniquely divisible cohomology groups,
then $R T_{l}(C^{\cdot}) = 0$.

Let $\varepsilon \colon S_{\fppf} \to S_{\et}$ be the morphism of sites
defined by the identity functor on the underlying categories.
Combining all the above and \eqref{eq: etale and Weil etale exact sequence},
we have for any $\mathcal{F} \in \Ab(S_{\fppf})$ a natural morphism and isomorphisms
	\begin{align*}
				R \Gamma_{W}(S, R \varepsilon_{\ast} \mathcal{F}) \tensor \Z_{l}
		&	\to
				R T_{l}(R \Gamma_{W}(S, R \varepsilon_{\ast} \mathcal{F}))[1]
		\\
		&	\isomfrom
				R T_{l}(R \Gamma_{\et}(S, R \varepsilon_{\ast} \mathcal{F}))[1]
		\\
		&	\cong
				R T_{l}(R \Gamma_{\fppf}(S, \mathcal{F}))[1]
		\\
		&	\cong
				R \Gamma_{\fppf}(S, R \underline{T}_{l}(\mathcal{F}))[1].
	\end{align*}
The first morphism is an isomorphism
if the groups $H_{W}^{\ast}(S, R \varepsilon_{\ast} \mathcal{F})$ are finitely generated.
If $\mathcal{F}$ is represented by a smooth group scheme,
then its fppf cohomology agrees with the \'etale cohomology (\cite[III, Rmk.\ 3.11 (b)]{Mil80}),
so $R \varepsilon_{\ast} \mathcal{F} = \mathcal{F}$.
With Prop.\ \ref{prop: finite generation}, we have the following.

\begin{Prop} \label{prop: integral models in derived setting}
	For any prime number $l$,
	we have canonical morphisms
		\begin{gather*}
					R \Gamma_{W}(S, \mathcal{A}) \tensor \Z_{l}
				\to
					R \Gamma_{\fppf}(S, R \underline{T}_{l}(\mathcal{A}))[1],
			\\
					R \Gamma_{W}(S, \mathcal{A}^{0}) \tensor \Z_{l}
				\to
					R \Gamma_{\fppf}(S, R \underline{T}_{l}(\mathcal{A}^{0}))[1].
		\end{gather*}
	The right-hand sides are canonically isomorphic to
	$R \Gamma_{\et}(S, R \underline{T}_{l}(\mathcal{A}))[1]$,
	$R \Gamma_{\et}(S, R \underline{T}_{l}(\mathcal{A}^{0}))[1]$, respectively,
	if $l \ne p$.
	These morphisms are isomorphisms if $\Sha(A)$ is finite.
\end{Prop}

We can give a more explicit description of the objects $R \underline{T}_{l}(\mathcal{A})$
and $R \underline{T}_{l}(\mathcal{A}^{0})$
and their fppf cohomology in some cases as follows.

\begin{Prop} \label{prop: integral models in abelian setting}
	For each place $v$, let $i_{v} \colon \Spec k(v) \into S$ denote the inclusion morphism.
	Assume that $l \ne p$ or $A$ has semistable reduction everywhere.
	\begin{enumerate}
	\item \label{item: cohom of Tate module: each degree}
		We have canonical ML-isomorphisms
			\begin{gather*}
							\underline{T}_{l}(\mathcal{A})
						\cong
							\underline{T}_{l}(\mathcal{A}^{0}),
					\quad
							R^{1} \underline{T}_{l}(\mathcal{A}^{0})
						=
							0,
				\\
						R^{1} \underline{T}_{l}(\mathcal{A})
					\cong
						\bigoplus_{v}
							i_{v \ast} \pi_{0}(\mathcal{A}_{v}) \tensor \Z_{l}
						\quad \text{(a constant system)},
			\end{gather*}
		\item \label{item: cohom of Tate module: finiteness}
			The group $H_{\fppf}^{n}(S, {}_{l^{i}}(\mathcal{A}^{0}))$ is finite
			for any $n \ge 0$ and $i \ge 1$.
		\item \label{item: cohom of Tate module: by naive l adic cohomology}
			The natural morphism from
			$H_{\fppf}^{n}(S, \underline{T}_{l}(\mathcal{A}^{0}))$ to
				$
					\invlim_{i}
						H_{\fppf}^{n}(S, {}_{l^{i}}(\mathcal{A}^{0}))
				$
			is an isomorphism.
		\item \label{item: cohom of Tate module: exact sequence}
			We have a long exact sequence
				\begin{align*}
							\cdots
					&	\to
							H_{\fppf}^{n}(S, \underline{T}_{l}(\mathcal{A}^{0}))
						\to
							H_{\fppf}^{n}(S, R \underline{T}_{l}(\mathcal{A}))
						\to
							\bigoplus_{v}
								H_{\et}^{n - 1}(k(v), \pi_{0}(\mathcal{A}_{v})) \tensor \Z_{l}
					\\
					&	\to
							H_{\fppf}^{n + 1}(S, \underline{T}_{l}(\mathcal{A}^{0}))
						\to
							\cdots.
				\end{align*}
	\end{enumerate}
\end{Prop}

Under the stated assumption,
the closed group subschemes ${}_{l^{i}}(\mathcal{A}^{0})$ of $\mathcal{A}^{0}$ over $S$
are quasi-finite, flat and separated over $S$ by \cite[III, Cor.\ C.9]{Mil06}.
They are \'etale if $l \ne p$,
and finite over the locus of $S$ where $A$ has good reduction.

\begin{proof}
	\eqref{item: cohom of Tate module: each degree}
	We have an exact sequence
		\[
				0
			\to
				\mathcal{A}^{0}
			\to
				\mathcal{A}
			\to
				\bigoplus_{v}
					i_{v \ast} \pi_{0}(\mathcal{A}_{v})
			\to
				0
		\]
	in $\Ab(S_{\fppf})$.
	Under the stated assumption, the multiplication by $l^{i}$ on $\mathcal{A}^{0}$ is faithfully flat
	for any $i$ by \cite[III, Cor.\ C.9]{Mil06}.
	Hence $R^{n} \underline{T}_{l}(\mathcal{A}^{0}) = 0$ for $n \ge 1$
	by \eqref{eq: first derived functor of formal Tate module}.
	Since the component groups are finite,
	we know that $\underline{T}_{l}(i_{v \ast} \pi_{0}(\mathcal{A}_{v}))$ is ML-zero and
	$R^{1} \underline{T}_{l}(i_{v \ast} \pi_{0}(\mathcal{A}_{v}))$ is ML-isomorphic to
	$i_{v \ast} \pi_{0}(\mathcal{A}_{v}) \tensor \Z_{l}$
	by \eqref{eq: first derived functor of formal Tate module}.
	This implies the result.
	
	\eqref{item: cohom of Tate module: finiteness}
	Consider the Hochschild-Serre spectral sequence
		\[
				E_{2}^{s t}
			=
				H_{\et}^{s}(\F_{q}, H_{\et}^{t}(\closure{S}, \closure{\mathcal{A}}))
			\Longrightarrow
				H_{\et}^{s + t}(S, \mathcal{A}).
		\]
	Using Prop.\ \ref{prop: geometric cohomology},
	we know that the kernels and cokernels of multiplication by $l^{i}$ on the $E_{2}$-terms
	are finite for all $i$.
	Hence the same is true for $H_{\et}^{n}(S, \mathcal{A})$
	and therefore for $H_{\et}^{n}(S, \mathcal{A}^{0})$.
	The long exact sequence associated with the sequence
	$0 \to {}_{l^{i}}(\mathcal{A}^{0}) \to \mathcal{A}^{0} \overset{l^{i}}{\to} \mathcal{A}^{0} \to 0$
	in $\Ab(S_{\fppf})$ gives an exact sequence
		\[
				0
			\to
				H_{\et}^{n - 1}(S, \mathcal{A}^{0}) \tensor \Z / l^{i} \Z
			\to
				H_{\fppf}^{n}(S, {}_{l^{i}}(\mathcal{A}^{0}))
			\to
				{}_{l^{i}}
				H_{\et}^{n}(S, \mathcal{A}^{0})
			\to
				0.
		\]
	Thus the middle term is also finite.
	
	\eqref{item: cohom of Tate module: by naive l adic cohomology}
	The previous assertion implies that the first derived inverse limit
	$\invlim_{i}^{1} H_{\fppf}^{n}(S, {}_{l^{i}}(\mathcal{A}^{0}))$ is zero.
	Hence the result follows from \cite[(3.1)]{Jan88}.
	
	\eqref{item: cohom of Tate module: exact sequence}
	The assertions above and the long exact sequence for $H_{\fppf}^{\ast}$ give the result.
\end{proof}

\begin{Cor}
	Assume that $l \ne p$ or $A$ has semistable reduction everywhere.
	We have a canonical homomorphism
		\[
				H_{W}^{n}(S, \mathcal{A}^{0}) \tensor \Z_{l}
			\to
				\invlim_{i}
					H_{\fppf}^{n + 1}(S, {}_{l^{i}}(\mathcal{A}^{0}))
		\]
	for any $n$.
	It is an isomorphism if $\Sha(A)$ is finite.
\end{Cor}

Note that if $l = p$ and $A$ has non-semistable reduction at some $v$,
then the multiplication by $p$ on $\mathcal{A}$ and $\mathcal{A}^{0}$ has a non-flat kernel
and a non-representable cokernel.


\end{document}